\documentclass[11pt]{article}

\usepackage[english]{babel}
\usepackage[utf8]{inputenc}
\usepackage[T1]{fontenc}
\usepackage{csquotes}

\usepackage{fullpage}
\usepackage[font={small},format=plain,labelfont=bf]{caption}
\usepackage{authblk}

\usepackage[numbers,sort]{natbib}

\usepackage{amsmath,amssymb,amsthm,mathtools}

\usepackage{graphics}

\usepackage{xcolor}
\definecolor{linkcolor}{rgb}{0,0.2,0.6}
\usepackage{hyperref}
\hypersetup{colorlinks,breaklinks,urlcolor=linkcolor,linkcolor=linkcolor,citecolor=linkcolor}

\usepackage{enumerate}

\usepackage{xfrac}

\usepackage{tikz-cd}
\newlength{\ptsize}
\makeatletter
\renewcommand{\fnum@figure}{Fig. \thefigure}
\makeatother

\newcommand{\cN}{\mathcal{N}}
\newcommand{\cT}{\mathcal{T}}

\newcommand{\ba}{\backslash}

\newtheorem{theorem}{Theorem}[section]
\newtheorem{lemma}[theorem]{Lemma}
\newtheorem{sublemma}{}[theorem]

\theoremstyle{definition}

\theoremstyle{remark}

\newcommand\blfootnote[1]{%
  \begingroup
  \renewcommand\thefootnote{}\footnote{#1}%
  \addtocounter{footnote}{-1}%
  \endgroup
}

\begin{document}

\title{A sharp lower bound for the number of phylogenetic trees displayed by a tree-child network \\
{\normalsize \em Dedicated to Andreas Dress in appreciation of his contributions to the phylogenetics community}}

\author[1$\ast$]{Charles Semple}
\author[2]{Kristina Wicke}

\affil[1]{School of Mathematics and Statistics, University of Canterbury, Christchurch, New Zealand, charles.semple@canterbury.ac.nz}
\affil[2]{Department of Mathematical Sciences, New Jersey Institute of Technology, Newark, NJ, USA, kristina.wicke@njit.edu}

\date{\today}

\maketitle

\begin{abstract}
A normal (phylogenetic) network with $k$ reticulations displays $2^k$ phylogenetic trees. In this paper, we establish an analogous result for tree-child (phylogenetic) networks with no underlying $3$-cycles. In particular, we show that a tree-child network with $k\ge 2$ reticulations and no underlying $3$-cycles displays at least $2^{\sfrac{k}{2}}$ phylogenetic trees if $k$ is even and at least $\frac{3}{2\sqrt{2}}2^{\sfrac{k}{2}}$ if $k$ is odd. Moreover, we show that these bounds are sharp and characterise the tree-child networks that attain these bounds.
\end{abstract}


\blfootnote{$^\ast$Corresponding author}

\section{Introduction}
\label{sec:introduction}

Understanding the evolutionary history of a collection of present-day species is a central goal in biology, and rooted phylogenetic trees have traditionally been used for this purpose. However, evolution is not always strictly tree-like. Reticulate evolutionary events, such as hybridisation and lateral gene transfer, violate the assumptions underlying phylogenetic trees and instead require a more general model, (rooted) phylogenetic networks, to accurately represent evolutionary history.

Although species-level evolution can be non-tree-like, the evolution of individual genes is typically assumed to follow a tree-like pattern. As a result, a phylogenetic network is often viewed as an amalgamation of gene trees. This viewpoint leads to the notion of a rooted phylogenetic tree displayed (intuitively, embedded) by a phylogenetic network. Several algorithms have been implemented to compute a phylogenetic network that displays a given collection of rooted phylogenetic trees. These algorithms include the Autumn algorithm~\cite{huson2018}, TreeKnit~\cite{barratcharlaix2022}, ALTS~\cite{zhang2023}, FHyNCH~\cite{bernardini2024}, and PhyloFusion~\cite{zhang2024}. Relatedly, there has been substantial work on questions such as whether a phylogenetic network is (uniquely) determined by its displayed rooted phylogenetic trees~\cite{gambette2011,willson2011,linz2020}, whether a particular rooted phylogenetic tree is displayed by a phylogenetic network~\cite{kanj2008,iersel2010}, and whether the number of rooted phylogenetic trees displayed by a given phylogenetic network can be computed in polynomial time~\cite{linz2013}. It is the last of these questions that is the attention of this paper. In general, it is \#P-complete to count the number of rooted binary phylogenetic trees displayed by a rooted binary phylogenetic network~\cite{linz2013} and, despite some recent progress~\cite{basire2025}, it remains an open problem on whether this computational hardness extends to counting the number of rooted binary phylogenetic trees displayed by a tree-child network, a particular, but well studied, type of phylogenetic network. In this paper, we focus on obtaining a sharp lower bound for the number of rooted binary phylogenetic trees displayed by a tree-child network. We complete the introduction by stating the main result of the paper. Formal definitions are given in the next section.

Let $\cN$ be a rooted binary phylogenetic network on $X$, and suppose that $\cN$ has $k$~reticulations. It is well-known that if $\cN$ is normal, then $\cN$ displays exactly $2^k$ rooted binary phylogenetic $X$-trees \cite{iersel2010, willson2012}. (This is the maximum possible number of rooted binary phylogenetic trees displayed by~$\cN$.) However, if $\cN$ is tree-child and we allow $\cN$ to have underlying $3$-cycles, then $\cN$ could have many reticulations but still display only one rooted binary phylogenetic $X$-tree. What can we say if $\cN$ is tree-child and has no underlying $3$-cycles? The number of rooted binary phylogenetic trees displayed can still be strictly less than $2^k$. But how much less?  In this paper, we establish the following theorem, the main result of the paper. In the statement of the theorem, note that a rooted binary tree-child network with $n$ leaves has at most $n-1$ reticulations and an octopus is a particular type of tree-child network that we describe in the next section. Also, for a rooted binary phylogenetic network $\cN$, we let $T(\cN)$ denote the set of (distinct) rooted binary phylogenetic trees displayed by~$\cN$.

\begin{theorem}
Let $\cN$ be a rooted binary tree-child network with $n$ leaves, $0\le k\le n-1$ reticulations, and no underlying $3$-cycles. If $k=0$, then $|T(\cN)|=1$, while if $k=1$, then $|T(\cN)|=2$. Furthermore, if $k\ge 2$, then
\begin{enumerate}[{\rm (i)}]
\item $|T(\cN)|\ge 2^{\sfrac{k}{2}}$ if $k$ is even, and

\item $|T(\cN)| \ge \frac{3}{2\sqrt{2}}2^{\sfrac{k}{2}}$ if $k$ is odd.
\end{enumerate}
Moreover, for all $k\ge 2$, we have that $|T(\cN)|=2^{\sfrac{k}{2}}$ and $k$ is even (respectively, $|T(\cN)| = \frac{3}{2\sqrt{2}}2^{\sfrac{k}{2}}$ and $k$ is odd) if and only if $\cN$ is an octopus.
\label{main}
\end{theorem}

The paper is organised as follows. In the next section, we give some necessary definitions that clarify the terminology in the statement of Theorem~\ref{main} and are used throughout the rest of the paper. Section~\ref{lemmas} establishes some preliminary lemmas, while Section~\ref{proof} consists of the proof of Theorem~\ref{main}.

\section{Preliminaries}
\label{sec:preliminaries}

Throughout the paper $X$ denotes a non-empty finite set.

\paragraph*{Phylogenetic networks.}
A {\em rooted binary phylogenetic network on $X$} is a rooted acyclic directed graph with no parallel arcs such that
\begin{enumerate}[(i)]
\item the (unique) root has in-degree zero and out-degree two,

\item the set of vertices of out-degree zero is $X$,

\item all other vertices have either in-degree one and out-degree two, or in-degree two and out-degree one.
\end{enumerate}
For technical reasons, if $|X|=1$, then we allow $\cN$ to consist of the single vertex in $X$. The set $X$ is call the {\em leaf set} of $\cN$. The vertices of in-degree one and out-degree two are {\em tree vertices}, while the vertices of in-degree two and out-degree one are {\em reticulations}. The arcs directed into a reticulation are called {\em reticulation arcs}; otherwise, an arc is a {\em tree arc}. If $(u, v)$ is a reticulation arc in $\cN$ and there is a directed path from $u$ to $v$ distinct from the path consisting of $(u, v)$, then $(u, v)$ is a {\em shortcut}. A reticulation $v$ {\em normal} if neither reticulation arc directed into $v$ is a shortcut. A {\em $2$-connected component} of $\cN$ is a maximal (underlying) subgraph of $\cN$ that is connected and cannot be disconnected by deleting exactly one of its vertices. We call a 2-connected component {\em trivial} if it consists of a single edge, and {\em non-trivial} otherwise. Furthermore, for brevity, we call an underlying $3$-cycle of $\cN$ a {\em $3$-cycle}. A {\em rooted binary phylogenetic $X$-tree $\cT$} is a rooted binary phylogenetic network on $X$ with no reticulations. Since all phylogenetic networks and phylogenetic trees in this paper are rooted and binary, we will refer to a rooted binary phylogenetic network and a rooted binary phylogenetic tree as a {\em phylogenetic network} and a {\em phylogenetic tree}, respectively.

A phylogenetic network $\cN$ on $X$ is {\em tree-child} if each non-leaf vertex is the parent of a tree vertex or a leaf. Equivalently, a phylogenetic network $\cN$ is tree-child precisely if no tree vertex is the parent of two reticulations and no reticulation is the parent of another reticulation~\cite{semple2015}. As an immediate consequence of the definition, if $u$ is a vertex of a tree-child network $\cN$, then there is a directed path from $u$ to a leaf $\ell$ of $\cN$ such that except for $\ell$ and possibly $u$, every vertex on the path is a tree vertex. We call such a path a {\em tree path (for $u$)}. As a result of this tree-path property, if $\cN$ is a tree-child network with $n$ leaves, then $\cN$ has at most $n-1$ reticulations, and this bound is sharp~\cite{cardona2009}. Also observe that if $C$ is a $3$-cycle of a tree-child network, then the arc set of $C$ consists of two reticulation arcs directed into the same reticulation, one of which is a shortcut, and a tree arc. A tree-child network is {\em normal} if it has no shortcuts. To illustrate, in Fig.~\ref{fig:display}(i), $\cN$ is a tree-child network, but it is not normal as the arc $(u,v)$ is a shortcut. As with all other figures in the paper, arcs are directed down the page. It is directly because of shortcuts that the number of phylogenetic trees displayed by a tree-child network with $k$ reticulations is not necessarily $2^k$.

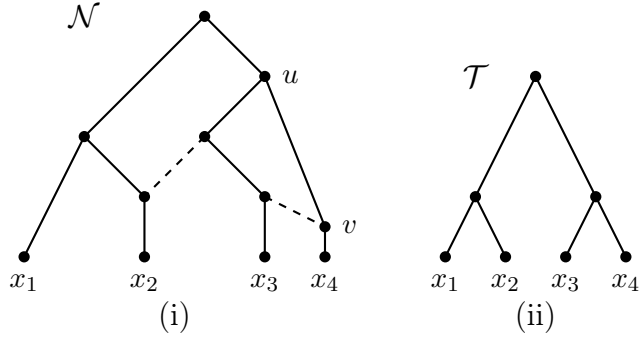
\begin{figure}[t!]
    \centering
    \begin{tikzpicture}[thick,scale=0.8]
    
            \node[fill=black,circle,inner sep=1.5pt, label=below: {$x_1$} ] at (0.5,1){};
            \node[fill=black,circle,inner sep=1.5pt, label=below: {$x_2$} ] at (2.5,1){};
            \node[fill=black,circle,inner sep=1.5pt, label=below: {$x_3$} ] at (4.5,1){};
            \node[fill=black,circle,inner sep=1.5pt, label=below: {$x_4$} ] at (5.5,1){};
            \node[fill=black,circle,inner sep=1.5pt]  at (3.5,5){};
            \node[fill=black,circle,inner sep=1.5pt] at (1.5,3){};
            \node[fill=black,circle,inner sep=1.5pt] at (2.5,2){};
            \node[fill=black,circle,inner sep=1.5pt] at (3.5,3){};
            \node[fill=black,circle,inner sep=1.5pt, label=right: {$u$}] at (4.5,4){};
            \node[fill=black,circle,inner sep=1.5pt] at (4.5,2){};
            \node[fill=black,circle,inner sep=1.5pt, label=right: {$v$}] at (5.5,1.5){};
            \draw(3.5,5)--(1.5,3);
            \draw(3.5,5)--(4.5,4);
            \draw(1.5,3)--(0.5,1);
            \draw(1.5,3)--(2.5,2);
            \draw[dashed](3.5,3)--(2.5,2);
            \draw(3.5,3)--(4.5,2);
            \draw[dashed](4.5,2)--(5.5,1.5);
            \draw(2.5,2)--(2.5,1);
            \draw(4.5,2)--(4.5,1);
            \draw(5.5,1.5)--(5.5,1);
            \draw(4.5,4)--(3.5,3);
            \draw(4.5,4)--(5.5,1.5);
            \node[align=left] at (1.5,5) {\large $\cN$};
             \node[align=left] at (3,0) {\large (i)};
                          
     		\node[fill=black,circle,inner sep=1.5pt, label=below: {$x_1$} ] at (7.5,1){};
            \node[fill=black,circle,inner sep=1.5pt, label=below: {$x_2$} ] at (8.5,1){};
            \node[fill=black,circle,inner sep=1.5pt, label=below: {$x_3$} ] at (9.5,1){};
            \node[fill=black,circle,inner sep=1.5pt, label=below: {$x_4$} ] at (10.5,1){};
            \node[fill=black,circle,inner sep=1.5pt]  at (9,4){};
     		\node[fill=black,circle,inner sep=1.5pt] at (8,2){};
     		\node[fill=black,circle,inner sep=1.5pt] at (10,2){};
     		\draw(9,4)--(7.5,1);
     		\draw(9,4)--(10.5,1);
     		\draw(8,2)--(8.5,1);
     		\draw(10,2)--(9.5,1);
     		\node[align=left] at (8,4) {\large $\cT$};
     		 \node[align=left] at (9,0) {\large (ii)};
       \end{tikzpicture}
    \caption{(i) A tree-child network $\cN$ on $X = \{x_1, x_2, x_3, x_4\}$ and (ii) a phylogenetic $X$-tree $\cT$ displayed  by $\cN$.}
    \label{fig:display}
\end{figure}

A lemma that we will frequently and freely use is the following \cite{doecker2021}.

\begin{lemma}
Let $\cN$ be a tree-child network with root $\rho$ and let $e=(u, v)$ be a reticulation arc of $\cN$. Then the phylogenetic network obtained from $\cN$ by deleting $e$ and either
\begin{enumerate}[{\rm (i)}]
\item suppressing the two resulting vertices of in-degree one and out-degree one if $u\neq \rho$, or

\item suppressing the resulting vertex of in-degree one and out-degree one, and deleting $u$ if $u=\rho$
\end{enumerate} 
is tree-child.
\label{deletion}
\end{lemma}

\noindent To ease reading, for a tree child network $\cN$ and reticulation arc $e$ of $\cN$, we denote by $\cN\ba e$ the operation of deleting $e$ and appropriately applying either (i) or (ii) of Lemma~\ref{deletion}. We next describe two particular types of tree-child networks that are central to this paper.

\paragraph*{Tight caterpillar ladders and octopuses.}
Let $\cN$ be a tree-child network with vertex set $\{\ell_0, \ell_1, \ell_2, \ell_3\}\cup \big\{u_i, u'_i, v_i: i\in \{1, 2, 3\}\big\}$. We call $\cN$ a {\em $3$-tight caterpillar ladder} if the arc set of $\cN$ is
$$\{(u'_3, u'_2), (u'_2, u_3), (u_3, u'_1), (u'_1, u_2), (u_2, u_1), (u_1, \ell_0)\}\cup \big\{(u'_i, v_i), (u_i, v_i), (v_i, \ell_i): i\in \{1, 2, 3\}\big\}.$$
Note that $\{\ell_0, \ell_1, \ell_2, \ell_3\}$ is the leaf set of $\cN$. The reticulation arcs $(u_1, v_1)$, $(u_2, v_2)$, $(u'_1, v_1)$, $(u_3, v_3)$, $(u'_2, v_2)$, and $(u'_3, v_3)$ are the {\em rungs} of the $3$-tight caterpillar ladder. Under this ordering, we refer to these rungs as the {\em $i$-th rung} so that, for example, $(u_1, v_1)$ and $(u'_3, v_3)$ are the {\em first} and {\em last} rungs of $\cN$, respectively. Furthermore, a tree-child network is a {\em $2$-tight caterpillar ladder} if it can be obtained from a $3$-tight caterpillar by deleting, in this instance, $u'_3$, $v_3$, and $\ell_3$, and suppressing the resulting vertex of in-degree one and out-degree one. Here, for example, the {\em first} and {\em last} rungs are the arcs $(u_1, v_1)$ and $(u'_2, v_2)$, respectively. For illustration, a $2$-tight and a $3$-tight caterpillar ladder are depicted in Fig.~\ref{fig:ladder}.

\begin{figure}[t!]
    \centering
    \begin{tikzpicture}[thick,scale=1]
    \node[fill=black,circle,inner sep=1.5pt,label={[right, yshift=2]{$u_2'$}}]  at (1,5){};
    \node[fill=black,circle,inner sep=1.5pt,label={[right, yshift=2]{$u_1'$}}]  at (1,3.5){};
    \node[fill=black,circle,inner sep=1.5pt,label={[right, yshift=-6]{$u_2$}}]  at (1,3){};
    \node[fill=black,circle,inner sep=1.5pt,label={[right, yshift=2]{$u_1$}}]  at (1,1.5){};
    \node[fill=black,circle,inner sep=1.5pt,label={[left, yshift=0]{$v_2$}}]  at (0,2.5){};
    \node[fill=black,circle,inner sep=1.5pt,label={[right, yshift=0]{$v_1$}}]  at (2,1){};
    \node[fill=black,circle,inner sep=1.5pt, label=below: {$\ell_0$} ] at (1,1){};
    \node[fill=black,circle,inner sep=1.5pt, label=below: {$\ell_1$} ] at (2,0.5){};
    \node[fill=black,circle,inner sep=1.5pt, label=below: {$\ell_2$} ] at (0,2){};

	\draw(1,5)--(1,1);
    \draw(1,3)--(0,2.5);
    \draw(0,2.5)--(0,2);
    \draw(1,1.5)--(2,1);
    \draw(2,1)--(2,0.5);
    \draw (1,5) to[bend left=-30] (0,2.5);
    \draw (1,3.5) to[bend left=30] (2,1);

    \node[fill=black,circle,inner sep=1.5pt,label={[right, yshift=2]{$u_3'$}}]  at (5,6.5){};
    \node[fill=black,circle,inner sep=1.5pt,label={[right, yshift=2]{$u_2'$}}]  at (5,5){};
    \node[fill=black,circle,inner sep=1.5pt,label={[right, yshift=2]{$u_3$}}]  at (5,4.5){};
    \node[fill=black,circle,inner sep=1.5pt,label={[right, yshift=2]{$u_1'$}}]  at (5,3.5){};
    \node[fill=black,circle,inner sep=1.5pt,label={[right, yshift=-6]{$u_2$}}]  at (5,3){};
    \node[fill=black,circle,inner sep=1.5pt,label={[right, yshift=2]{$u_1$}}]  at (5,1.5){};
    \node[fill=black,circle,inner sep=1.5pt,label={[right, yshift=0]{$v_3$}}]  at (6,4){};
    \node[fill=black,circle,inner sep=1.5pt,label={[left, yshift=0]{$v_2$}}]  at (4,2.5){};
    \node[fill=black,circle,inner sep=1.5pt,label={[right, yshift=0]{$v_1$}}]  at (6,1){};
    \node[fill=black,circle,inner sep=1.5pt, label=below: {$\ell_0$} ] at (5,1){};
    \node[fill=black,circle,inner sep=1.5pt, label=below: {$\ell_1$} ] at (6,0.5){};
    \node[fill=black,circle,inner sep=1.5pt, label=below: {$\ell_2$} ] at (4,2){};
    \node[fill=black,circle,inner sep=1.5pt, label=below: {$\ell_3$} ] at (6,3.5){};
     
    \draw(5,6.5)--(5,1);
    \draw(5,4.5)--(6,4);
    \draw(6,4)--(6,3.5);
    \draw(5,3)--(4,2.5);
    \draw(4,2.5)--(4,2);
    \draw(5,1.5)--(6,1);
    \draw(6,1)--(6,0.5);
    \draw (5,6.5) to[bend left=30] (6,4);
    \draw (5,5) to[bend left=-30] (4,2.5);
    \draw (5,3.5) to[bend left=30] (6,1);
    \end{tikzpicture}
    
     \caption{A $2$-tight (left) and a $3$-tight (right) caterpillar ladder.}
    \label{fig:ladder}
\end{figure}
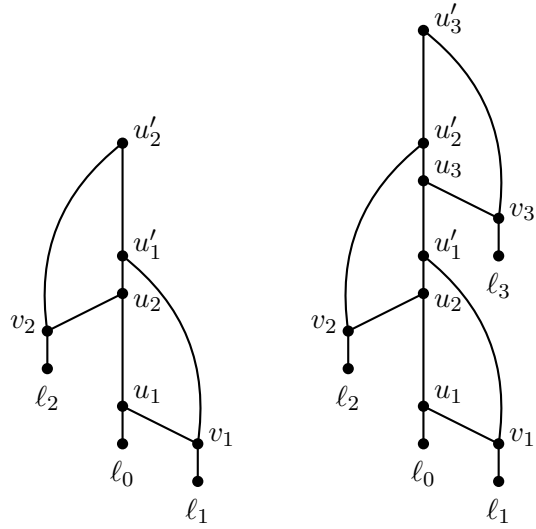

Let $k\in \{2, 3\}$. The {\em core} of a $k$-tight caterpillar ladder consists of its non-pendant arcs. Furthermore, let $\cN$ be a $k$-tight caterpillar ladder and let $\cN'$ be a tree-child network. We say that $\cN$ is a {\em $k$-tight caterpillar ladder of $\cN'$} if, up to isomorphism, the core of $\cN$ can be obtained from $\cN'$ by deleting vertices and arcs.

Now let $\cN$ be a tree-child network on $X$ with $n$ leaves and $k$ reticulations, where $k\neq 1$. We call $\cN$ an {\em octopus} if either $k$ is even and every non-trivial $2$-connected component of $\cN$ is the core of a $2$-tight caterpillar ladder, or $k$ is odd, exactly one non-trivial $2$-connected component of $\cN$ is the core of a $3$-tight caterpillar ladder, and every other non-trivial $2$-connected component of $\cN$ is the core of a $2$-tight caterpillar ladder. To illustrate, in Fig.~\ref{fig:octopus}, $\cN$ and $\cN'$ are both octopuses with $10$ leaves and $7$ reticulations.

\begin{figure}[t!]
    \centering
    \begin{tikzpicture}[thick,scale=0.8]
    
    \node[fill=black,circle,inner sep=1.5pt, label=below: {$x_1$} ] at (0.5,1){};
    \node[fill=black,circle,inner sep=1.5pt, label=below: {$x_2$} ] at (1.5,1.5){};
    \node[fill=black,circle,inner sep=1.5pt, label=below: {$x_3$} ] at (2,2){};
    \node[fill=black,circle,inner sep=1.5pt, label=below: {$x_4$} ] at (3,2){};
    \node[fill=black,circle,inner sep=1.5pt, label=below: {$x_5$} ] at (3,5.5){};
    \node[fill=black,circle,inner sep=1.5pt, label=below: {$x_6$} ] at (4,5.5){};
    \node[fill=black,circle,inner sep=1.5pt, label=below: {$x_7$} ] at (4,2){};
    \node[fill=black,circle,inner sep=1.5pt, label=below: {$x_8$} ] at (5,1){};
    \node[fill=black,circle,inner sep=1.5pt, label=below: {$x_9$} ] at (6,0.5){};
    \node[fill=black,circle,inner sep=1.5pt, label=below: {$x_{10}$} ] at (6,3.5){};
 	\node[fill=black,circle,inner sep=1.5pt]  at (0.5,1.5){};
 	\node[fill=black,circle,inner sep=1.5pt]  at (1.5,2){};
 	\node[fill=black,circle,inner sep=1.5pt]  at (2.5,2.5){};
 	\node[fill=black,circle,inner sep=1.5pt]  at (2.5,3){};
 	\node[fill=black,circle,inner sep=1.5pt]  at (1.5,3.5){};
 	\node[fill=black,circle,inner sep=1.5pt]  at (1.5,4){};
 	\node[fill=black,circle,inner sep=1.5pt]  at (1.5,5.5){};
 	\node[fill=black,circle,inner sep=1.5pt]  at (3.5,6){};
 	\node[fill=black,circle,inner sep=1.5pt]  at (2.5,6.5){};
 	\node[fill=black,circle,inner sep=1.5pt]  at (3.5,7){};
 	\node[fill=black,circle,inner sep=1.5pt]  at (3.5,8.5){};
 	\node[fill=black,circle,inner sep=1.5pt]  at (3.5,9){};
 	\node[fill=black,circle,inner sep=1.5pt]  at (3.5,10.5){};
 	\node[fill=black,circle,inner sep=1.5pt]  at (4.5,8){};
 	\node[fill=black,circle,inner sep=1.5pt]  at (5,6.5){};
 	\node[fill=black,circle,inner sep=1.5pt]  at (5,5){};
 	\node[fill=black,circle,inner sep=1.5pt]  at (5,4.5){};
  	\node[fill=black,circle,inner sep=1.5pt]  at (5,3.5){};
  	\node[fill=black,circle,inner sep=1.5pt]  at (5,3){};
   	\node[fill=black,circle,inner sep=1.5pt]  at (5,1.5){};
    \node[fill=black,circle,inner sep=1.5pt]  at (5,1){};
  	\node[fill=black,circle,inner sep=1.5pt]  at (6,4){};
   	\node[fill=black,circle,inner sep=1.5pt]  at (4,2.5){};
    \node[fill=black,circle,inner sep=1.5pt]  at (6,1){};
    \node[fill=black,circle,inner sep=1.5pt]  at (3.5,9){};
    \draw(3.5,10.5)--(3.5,6);
    \draw(3.5,6)--(3,5.5);
    \draw(3.5,6)--(4,5.5);
    \draw(3.5,8.5)--(4.5,8);
    \draw(3.5,7)--(2.5,6.5);
    \draw(2.5,6.5)--(1.5,5.5);
    \draw(4.5,8)--(5,6.5);
    \draw (3.5,10.5) to[bend left=30] (4.5,8);
    \draw (3.5,9) to[bend left=-30] (2.5,6.5);
    \draw(1.5,5.5)--(1.5,1.5);
    \draw(1.5,3.5)--(2.5,3);
    \draw(1.5,2)--(0.5,1.5);
    \draw(0.5,1.5)--(0.5,1);
    \draw(1.5,2)--(1.5,1.5);
    \draw(2.5,3)--(2.5,2.5);
    \draw(2.5,2.5)--(2,2);
    \draw(2.5,2.5)--(3,2);
    \draw (1.5,5.5) to[bend left=30] (2.5,3);
    \draw (1.5,4) to[bend left=-30] (0.5,1.5);
    \draw(5,6.5)--(5,1);
    \draw(5,4.5)--(6,4);
    \draw(6,4)--(6,3.5);
    \draw(5,3)--(4,2.5);
    \draw(4,2.5)--(4,2);
    \draw(5,1.5)--(6,1);
    \draw(6,1)--(6,0.5);
    \draw (5,6.5) to[bend left=30] (6,4);
    \draw (5,5) to[bend left=-30] (4,2.5);
    \draw (5,3.5) to[bend left=30] (6,1);
    \node[align=left] at (2.5,10.5) {\Large $\cN$};
        
    \node[fill=black,circle,inner sep=1.5pt, label=below: {$x_1$} ] at (8,1){};
    \node[fill=black,circle,inner sep=1.5pt, label=below: {$x_2$} ] at (9,1.5){};
    \node[fill=black,circle,inner sep=1.5pt, label=below: {$x_3$} ] at (10,2.5){};
    \node[fill=black,circle,inner sep=1.5pt, label=below: {$x_4$} ] at (10.5,1){};
    \node[fill=black,circle,inner sep=1.5pt, label=below: {$x_5$} ] at (11.5,1.5){};
    \node[fill=black,circle,inner sep=1.5pt, label=below: {$x_6$} ] at (12.5,2.5){};
    \node[fill=black,circle,inner sep=1.5pt, label=below: {$x_7$} ] at (13.5,3){};
    \node[fill=black,circle,inner sep=1.5pt, label=below: {$x_8$} ] at (14.5,2){};
    \node[fill=black,circle,inner sep=1.5pt, label=below: {$x_9$} ] at (15.5,1.5){};
    \node[fill=black,circle,inner sep=1.5pt, label=below: {$x_{10}$} ] at (15.5,4.5){};
    \node[fill=black,circle,inner sep=1.5pt]  at (8,1.5){};
    \node[fill=black,circle,inner sep=1.5pt]  at (9,2){};
    \node[fill=black,circle,inner sep=1.5pt]  at (9,3.5){};
    \node[fill=black,circle,inner sep=1.5pt]  at (10,3){};
    \node[fill=black,circle,inner sep=1.5pt]  at (9,4){};
    \node[fill=black,circle,inner sep=1.5pt]  at (9,5.5){};
    \node[fill=black,circle,inner sep=1.5pt]  at (10.5,1.5){};
    \node[fill=black,circle,inner sep=1.5pt]  at (11.5,1.5){};
    \node[fill=black,circle,inner sep=1.5pt]  at (11.5,2){};
    \node[fill=black,circle,inner sep=1.5pt]  at (12.5,3){};
    \node[fill=black,circle,inner sep=1.5pt]  at (11.5,3.5){};
    \node[fill=black,circle,inner sep=1.5pt]  at (11.5,4){};
    \node[fill=black,circle,inner sep=1.5pt]  at (11.5,5.5){};
    \node[fill=black,circle,inner sep=1.5pt]  at (10.25,7.5){};
    \node[fill=black,circle,inner sep=1.5pt]  at (12.5,10.5){};
    \node[fill=black,circle,inner sep=1.5pt]  at (14.5,7.5){};
    \node[fill=black,circle,inner sep=1.5pt]  at (14.5,6){};
    \node[fill=black,circle,inner sep=1.5pt]  at (14.5,5.5){};
    \node[fill=black,circle,inner sep=1.5pt]  at (15.5,5){};
    \node[fill=black,circle,inner sep=1.5pt]  at (14.5,4.5){};
    \node[fill=black,circle,inner sep=1.5pt]  at (14.5,4){};
    \node[fill=black,circle,inner sep=1.5pt]  at (13.5,3.5){};
    \node[fill=black,circle,inner sep=1.5pt]  at (14.5,2.5){};
    \node[fill=black,circle,inner sep=1.5pt]  at (15.5,2){};
    
    \draw(12.5,10.5)--(10.25,7.5);
    \draw(12.5,10.5)--(14.5,7.5);
    \draw(10.25,7.5)--(9,5.5);
    \draw(10.25,7.5)--(11.5,5.5);
    
    \draw(9,5.5)--(9,1.5);
    \draw(9,3.5)--(10,3);
    \draw(10,3)--(10,2.5);
    \draw(9,2)--(8,1.5);
    \draw(8,1.5)--(8,1);
    \draw (9,5.5) to[bend left=30] (10,3);
    \draw (9,4) to[bend left=-30] (8,1.5);
    
    \draw(11.5,5.5)--(11.5,1.5);
    \draw(11.5,3.5)--(12.5,3);
    \draw(11.5,2)--(10.5,1.5);
    \draw(10.5,1.5)--(10.5,1);
    \draw(12.5,3)--(12.5,2.5);
    \draw (11.5,5.5) to[bend left=30] (12.5,3);
    \draw (11.5,4) to[bend left=-30] (10.5,1.5);
    
    \draw(14.5,7.5)--(14.5,2);
    \draw (14.5,7.5) to[bend left=30] (15.5,5);
    \draw (14.5,6) to[bend left=-30] (13.5,3.5);
    \draw (14.5,4.5) to[bend left=30] (15.5,2);
    \draw(14.5,5.5)--(15.5,5);
    \draw(15.5,5)--(15.5,4.5);
    \draw(14.5,4)--(13.5,3.5);
    \draw(13.5,3.5)--(13.5,3);
    \draw(14.5,2.5)--(15.5,2);
    \draw(15.5,2)--(15.5,1.5);
    
     \node[align=left] at (10.25,10.5) {\Large $\cN'$};
   \end{tikzpicture}
   
    \caption{Two octopuses $\cN$ and $\cN'$ with $10$ leaves and $7$ reticulations.}
    \label{fig:octopus}
\end{figure}
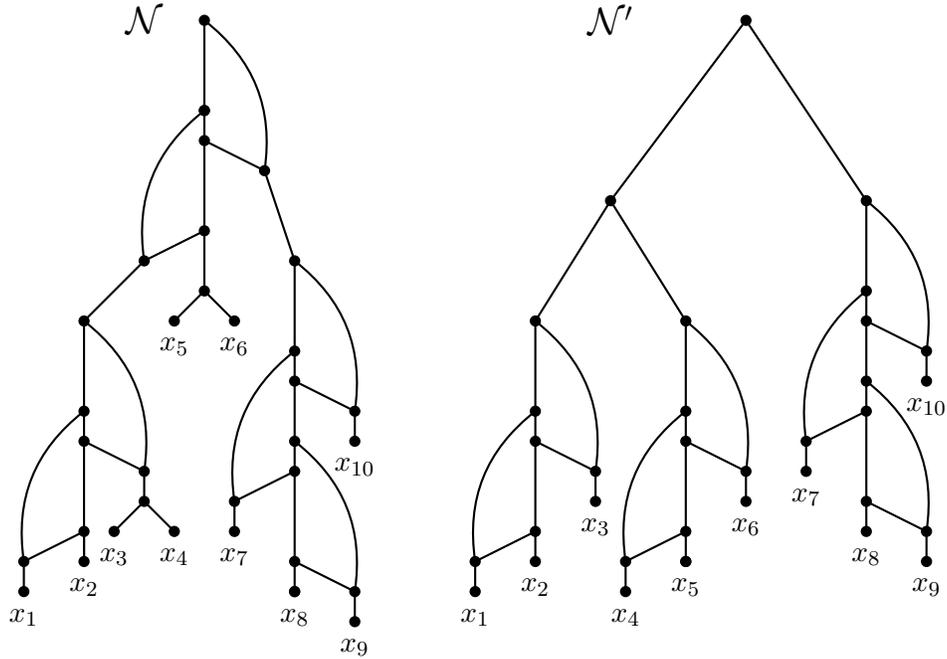

\paragraph*{Displaying.}
Let $\cN$ be a phylogenetic network on $X$ and let $\cT$ be a phylogenetic $X$-tree. We say that $\cN$ {\em displays} $\cT$ if a subdivision of $\cT$ can be obtained from $\cN$ by deleting vertices and arcs. Such a subdivision is an {\em embedding} of $\cT$ in $\cN$. If $\cN$ is a tree-child network and $\cT$ is a phylogenetic tree displayed by $\cN$, then every embedding of $\cT$ contains all of the tree arcs of $\cN$ and, for each reticulation $v$, exactly one reticulation arc of $\cN$ directed into $v$. Conversely, if $F$ is a subset of the arcs of $\cN$ that consists of all tree arcs and, for each reticulation $v$, exactly one reticulation arc directed into $v$, then $F$ is an embedding of a phylogenetic tree displayed by $\cN$ \cite{semple2015}. Thus to describe an embedding of $\cT$ in $\cN$ it suffices to specify the reticulation arcs of $\cN$ in the embedding. Such arcs are {\em used} by $\cT$. Also, as a reminder to the reader, we use $T(\cN)$ to denote the set of phylogenetic $X$-trees displayed by $\cN$. To illustrate the notion of display, in Fig.~\ref{fig:display}, $\cN$ displays $\cT$, where an embedding of $\cT$ in $\cN$ is shown as solid arcs. Note that there is one other distinct embedding of $\cT$ in $\cN$. 

Now let $\cN$ be a phylogenetic network on $X$. An arc $e$ of $\cN$ is {\em non-essential} if, for every phylogenetic $X$-tree $\cT$ in $T(\cN)$, there is an embedding of $\cT$ in $\cN$ that avoids $e$. The next lemma is a special case of a more general result established in \cite{linz2022}.

\begin{lemma}
Let $\cN$ be a tree-child network with $n$ leaves and $k$ reticulations, where $k\in \{2, 3\}$, and let $e$ be an arc of $\cN$. Then $e$ is non-essential if and only if $e$ is either the first or last rung of a $2$- or $3$-tight caterpillar ladder of $\cN$.
\label{tight}
\end{lemma}

\paragraph*{Phylogenetic trees.}
Let $\cT$ be a phylogenetic $X$-tree, and let $X'$ be a subset of $X$. The minimal subtree of $\cT$ connecting the elements in $X'$ is denoted by $\cT(X')$. Furthermore, the {\em restriction of $\cT$ to $X'$} is the phylogenetic $X'$-tree obtained from $\cT(X')$ by suppressing vertices of in-degree one and out-degree one. A subtree of $\cT$ is {\em pendant} if it can be obtained by deleting an edge of $\cT$, in which case the leaf set of this pendant subtree is a {\em cluster} of $\cT$. Furthermore, two phylogenetic $X$-trees $\cT_1$ and $\cT_2$ are {\em isomorphic} if there is a map $\varphi: V(\cT_1)\rightarrow V(\cT_2)$ such that, for all $x\in X$, we have $\varphi(x)=x$ and, if $(u, v)$ is an arc of $\cT_1$, then $(\varphi(u), \varphi(v))$ is an arc of $\cT_2$.

A {\em caterpillar} is a phylogenetic tree whose leaf set can be ordered $x_1, x_2, \ldots, x_n$ so that $x_1$ and $x_2$ have the same parent and, for all $i\in \{3, 4, \ldots, n\}$, the parent of $x_{i-1}$ is a child of the parent of $x_i$. We denote such a caterpillar by $(x_1, x_2, \ldots, x_n)$. A {\em double caterpillar} is a phylogenetic tree such that its maximal pendant subtrees are both caterpillars. If $(x_1, x_2, \ldots, x_i)$ and $(y_1, y_2, \ldots, y_j)$ are two such caterpillars, then we denote the double caterpillar by $\big\{(x_1, x_2, \ldots, x_i), (y_1, y_2, \ldots, y_j)\big\}$.

\section{Some lemmas}
\label{lemmas}

In this section we establish some lemmas that will be used in the proof of Theorem~\ref{main}. For the first lemma, recall that a tree-child network with $n$ leaves has at most $n-1$ reticulations. Let $n\geq 1$ and $0 \leq k \leq n-1$ be two non-negative integers. For all $n$ and $k$, set
$$t(n, k) =
\begin{cases}
1, & \mbox{if $k=0$;} \\
2, & \mbox{if $k=1$;} \\
2^{\sfrac{k}{2}}, & \mbox{if $k\geq 2$ and $k$ is even;} \\
\frac{3}{2\sqrt{2}}2^{\sfrac{k}{2}}, & \mbox{if $k\geq 3$ and $k$ is odd.}
\end{cases}$$
Observe that, for all $n$ and $k$, the value $t(n, k)$ is the bound given in the statement of Theorem~\ref{main}. The next lemma establishes some basic properties of the numbers $t(n, k)$. These properties are repeatedly used in the proof of Theorem~\ref{main}.

\begin{lemma}
The following identities hold:
\begin{enumerate}[{\rm (i)}]
\item For all $n\ge 2$, we have $t(n,1) = 2 \cdot t(n,0)$ and, for all $n\ge 3$, we have $t(n,2) = 2 \cdot t(n,0)$.

\item For all $n\ge 4$, we have $t(n,3) = t(n,2) + t(n,0) < 2 \cdot t(n,1) = t(n, 2) + t(n, 1)$.

\item For all $n\ge 4$ and $2\le k\le n-2$, we have $t(n, k) < t(n, k+1)$.

\item For all $n\ge 4$ and $3\le k\le n-1$, we have $t(n,k) < 4 \cdot t(n,k-3)$. 

\item For all $n\ge 5$ and $4\le k\le n-1$,
$$t(n, k) = 2\cdot t(n, k-2) < t(n, k-1) + t(n, k-2).$$

\item For all $n\ge 4$ and $3\le k\le n-1$, and $k$ is odd,
$$t(n, k) = t(n, k-1) + t(n, k-3)$$
while, for all $n\ge 5$ and $4\le k\le n-1$, and $k$ is even,
$$t(n, k) < t(n, k-1) + t(n, k-3).$$
\end{enumerate}
\label{tnk}
\end{lemma}

\begin{proof} The proof of (i) is trivial. For the proof of (ii), if $n\ge 4$, then
$$t(n, 3) = \tfrac{3}{2\sqrt{2}}2^{\sfrac{3}{2}} = 3 = t(n, 2) + t(n, 0) < 4 = 2\cdot t(n, 1) = t(n, 2) + t(n, 1).$$
For the proof of (iii), if $n\ge 4$, $2\le k\le n-2$, and $k$ is even, then
\begin{align*}
t(n, k) & = 2^{\sfrac{k}{2}}
< \tfrac{3}{2}\cdot 2^{\sfrac{k}{2}}
= \tfrac{3}{2\sqrt{2}}2^{\sfrac{k}{2}}\cdot \sqrt{2}
= \tfrac{3}{2\sqrt{2}}2^{\sfrac{(k+1)}{2}} = t(n, k+1)
\end{align*}
while if $n\ge 4$, $2\le k\le n-2$, and $k$ is odd, then
\begin{align*}
t(n, k) & = \tfrac{3}{2\sqrt{2}}2^{\sfrac{k}{2}}
=  3\cdot 2^{\sfrac{(k-3)}{2}} < 4\cdot 2^{\sfrac{(k-3)}{2}} = 2^{\sfrac{(k+1)}{2}} = t(n, k+1).
\end{align*}

Now consider (iv). If $n\ge 4$ and $k=3$, then $t(n, 3) = 3 < 4 = 4\cdot t(n, 0)$ and, if $n\ge 5$ and $k=4$, then $t(n, 4) = 4 < 8 = 4\cdot t(n, 1)$, so we may assume that $n\ge 6$ and $5\le k\le n-1$. If $k$ is odd, then
$$t(n, k) = \tfrac{3}{2\sqrt{2}}2^{\sfrac{k}{2}} = 3\cdot 2^{\sfrac{(k-3)}{2}} < 4\cdot 2^{\sfrac{(k-3)}{2}} < 4\cdot t(n, k-3),$$
while if $k$ is even, then
$$t(n, k) = 2^{\sfrac{k}{2}} = 2\cdot 2^{\sfrac{(k-2)}{2}} < 3\cdot 2^{\sfrac{(k-2)}{2}} = 4\cdot \tfrac{3}{2\sqrt{2}}2^{\sfrac{(k-3)}{2}} = 4\cdot t(n, k-3).$$

For the proof of (v), if $n\ge 5$, $4\le k\le n-1$, and $k$ is even, then, by (iii),
$$t(n, k) = 2^{\sfrac{k}{2}} = 2\cdot 2^{\sfrac{(k-2)}{2}} = 2\cdot t(n, k-2) < t(n, k-1) + t(n, k-2).$$
If $n\ge 5$, $4\le k\le n-1$, and $k$ is odd, then, by (iii),
$$t(n, k) = \tfrac{3}{2\sqrt{2}} 2^{\sfrac{k}{2}} = 2\cdot \tfrac{3}{2\sqrt{2}} 2^{\sfrac{(k-2)}{2}} = 2\cdot t(n, k-2) < t(n, k-1) + t(n, k-2).$$

Lastly, consider (vi). If $n\ge 4$ and $k=3$, then
$$t(n, k-1) + t(n, k-3) = t(n, 2) + t(n, 0) = t(n, 3).$$
Furthermore, if $n\ge 6$, $3\le k\le n-1$, and $k$ is odd, then, as $k-1$ and $k-3$ are even,
\begin{align*}
t(n, k-1) + t(n, k-3) & = 2^{\sfrac{(k-1)}{2}} + 2^{\sfrac{(k-3)}{2}} = 2^{\sfrac{(k-3)}{2}}\left(2 + 1 \right) \\
& = 3 \cdot 2^{\sfrac{(k-3)}{2}} = \frac{3}{2\sqrt{2}}2^{\sfrac{k}{2}} = t(n, k).
\end{align*}
If $n\ge 5$, $4\le k\le n-1$, and $k$ is even, then, as $k-1$ and $k-3$ are odd,
\begin{align*}
t(n, k-1) + t(n, k-3) & = \tfrac{3}{2\sqrt{2}} 2^{\sfrac{(k-1)}{2}} + \tfrac{3}{2\sqrt{2}} 2^{\sfrac{(k-3)}{2}} = \tfrac{3}{2\sqrt{2}} 2^{\sfrac{(k-3)}{2}} (2+1) \\
& > \tfrac{2^3}{2\sqrt{2}} 2^{\sfrac{(k-3)}{2}} = 2^{\sfrac{k}{2}} = t(n, k).
\end{align*}

\end{proof}

The next lemma establishes the number of distinct phylogenetic trees displayed by an octopus.

\begin{lemma} 
Let $\cN$ be an octopus with $n$ leaves and $k \neq 1$ reticulations. Then $|T(\cN)|=t(n, k)$.
\label{octopus}
\end{lemma}

\begin{proof}
The proof is by induction on $k$. Evidently, if $k=0$, then $| T(\cN)| = 1 = t(n,0)$. Furthermore, if $k=2$, then $\cN$ contains precisely one non-trivial $2$-connected component which is the core of a $2$-tight caterpillar ladder and it is easy to verify that $|T(\cN)| = 2 = t(n, 2)$. Similarly, if $k=3$, then $\cN$ contains precisely one non-trivial $2$-connected component which is the core of a $3$-tight caterpillar ladder and $|T(\cN)| = 3 = t(n, 3)$. Thus the lemma holds for $k\in \{0, 2, 3\}$.

Now suppose that $k\geq 4$, in which case $n\ge 5$, and the lemma holds for all octopuses with at most $k-1$ reticulations. Let $v$ be a reticulation of $\cN$ with the property that all paths starting at $v$ are tree paths, that is, there are no other reticulations among the descendants of $v$. As $\cN$ is an octopus, $v$ is a reticulation of either a $2$-tight or $3$-tight caterpillar ladder, say $\cN'$, of $\cN$.

Assume that $\cN'$ is a $2$-tight caterpillar ladder of $\cN$. Let $\cN_1$ be the tree-child network obtained from $\cN$ by deleting the third and last rungs of $\cN'$, and let $\cN_2$ be the tree-child network obtained from $\cN$ by deleting the first and last rungs of $\cN'$. Clearly, $\cN_1$ and $\cN_2$ are octopuses with $n$ leaves and $k-2$ reticulations. Moreover, observe that $T(\cN_1)\cap T(\cN_2)$ is empty and, if $\cT\in T(\cN)$, then either $\cT\in T(\cN_1)$ or $\cT\in T(\cN_2)$.  Thus, by induction and Lemma~\ref{tnk}(v),
\begin{align*}
|T(\cN)| & = |T(\cN_1)| + |T(\cN_2)| 
= t(n, k-2) + t(n, k-2) = t(n, k).
\end{align*}

Now assume that $\cN'$ is a $3$-tight caterpillar ladder of $\cN$, in which case, $k$ is odd and so $k\ge 5$. Let $\cN_1$, $\cN_2$, and $\cN_3$ be the octopuses on $k-3$ reticulations obtained from $\cN$ by deleting the third, fifth, and last rungs of $\cN'$, the first, fifth, and last rungs of $\cN'$, and the second, third, and last rungs of $\cN'$, respectively. It is easily verified that $T(\cN_i)\cap T(\cN_j) = \emptyset$ for all distinct $i, j\in \{1, 2, 3\}$. Furthermore, if $\cT \in T(\cN)$, then $\cT \in T(\cN_1) \cup T(\cN_2) \cup T(\cN_3)$. Therefore, as $k$ is odd, it follows by induction that
\begin{align*}
|T(\cN)| & = |T(\cN_1)| + |T(\cN_2)| + |T(\cN_3)| \\
& = t(n, k-3) + t(n, k-3) + t(n, k-3) \\
& = 3\cdot 2^{\sfrac{(k-3)}{2}} = \tfrac{3}{2\sqrt{2}}2^{\sfrac{k}{2}} = t(n, k).
\end{align*}
This completes the proof of the lemma.
\end{proof}

For the proof of Theorem~\ref{main}, we need to understand what happens if we delete a reticulation arc of a tree-child network and create a $3$-cycle. The next three lemmas consider tree-child networks and $3$-cycles. 

\begin{lemma}
Let $\cN$ be a tree-child network with no $3$-cycles, and let $e$ be a reticulation arc of~$\cN$. Suppose that $\cN\ba e$ has a $3$-cycle with reticulation arcs $f$ and $f'$. Then each of $\cN\ba \{e, f\}$ and $\cN\ba \{e, f'\}$ is tree-child and has no $3$-cycles.
\label{no3cycles}
\end{lemma}

\begin{proof}
Consider $\cN \ba e$ and denote the arcs of the $3$-cycle of $\cN\ba e$ as $f = (u_1, v)$, $f'=(u_2,v)$, and $h = (u_1, u_2)$. In particular, $v$ is a reticulation, and $f$ is a shortcut of $\cN$ and $\cN\ba e$. Note that $\cN$ and $\cN\ba e$ contain the arcs $f$ and $f'$, but $\cN$ does not contain $h$. Instead, $\cN$ contains two arcs, say $h_1 = (u_1, s)$ and $h_2 = (s, u_2)$, such that $e = (s, t)$ is the reticulation arc that is deleted to obtain $\cN \ba e$ from $\cN$. We now argue that both $\cN\ba \{e, f\}$ and $\cN\ba \{e, f'\}$ are tree-child and have no $3$-cycles.

First, consider $\cN\ba \{e, f\}$ and suppose that it contains a $3$-cycle, say $C$. Then $C$ contains the arc $(p,u_2)$, where $p$ is the unique parent of $u_1$ in $\cN$. The two remaining arcs of $C$ are reticulation arcs, say $(p,r)$ and $(u_2,r)$, where $r \neq v$ is a reticulation in  $\cN\ba \{e, f\}$ and also in $\cN$. Since there is also an arc $(u_2,v)$ in $\cN$ and $v$ is a reticulation, this implies that both children of $u_2$ are reticulations, contradicting the fact that $\cN$ is tree-child. Thus $\cN\ba \{e, f\}$ is tree-child and has no $3$-cycle.

Next, consider $\cN\ba \{e, f'\}$ and suppose that it contains a $3$-cycle, say $C'$. Then $C'$ contains the arc $(u_1,w)$, where $w \neq v$ is a child of $u_2$ in $\cN$. As $\cN$ is tree-child, $w$ is a tree vertex. Thus the two remaining arcs of $C'$ are reticulation arcs incident with the same reticulation, say $r'\neq v$. But then $(u_1, r')$ is an arc of $C'$ and so, as $r'$ is a reticulation of $\cN$, the vertex $u_1$ is the parent of two reticulations in $\cN$, a contradiction as $\cN$ is tree-child. Therefore $\cN\ba \{e, f'\}$ is tree-child and has no $3$-cycles, thereby completing the proof of the lemma.
\end{proof}

\begin{lemma}
Let $\cN$ be a tree-child network with no $3$-cycles, and let $e = (u,v)$ be a reticulation arc of $\cN$ such that, amongst all reticulation arcs of $\cN$, $u$ has minimum distance to the root of $\cN$. Then $\cN \ba e$ is tree-child and has no $3$-cycles. 
\label{tail}
\end{lemma}

\begin{proof}
If $\cN\ba e$ has a $3$-cycle, then the parent $p$ of $u$ in $\cN$ has a child that is a reticulation. But then $p$ is closer to the root of $\cN$ than $u$, a contradiction. Hence $\cN\ba e$ has no $3$-cycles. 
\end{proof}

The last lemma in this section is a technical lemma that is used in the inductive proof of Theorem~\ref{main}. Recall that a reticulation is normal if neither reticulation arc directed into it is a shortcut.

\begin{lemma}
Let $\cN$ be a tree-child network with $n$ leaves, $k\ge 2$ reticulations, and no $3$-cycles. Suppose that $\cN$ has a normal reticulation and, for all tree-child networks $\cN'$ with $n$ leaves, $k'< k$ reticulations, and no $3$-cycles, $|T(\cN')|\ge t(n, k')$. Then $|T(\cN)| > t(n, k)$.
\label{normal}
\end{lemma}

\begin{proof}
Let $v$ be a normal reticulation of $\cN$, and let $e_1$ and $e_2$ denote the reticulation arcs directed into $v$. The proof is partitioned into three cases depending on whether zero, one, or two of $\cN\ba e_1$ and $\cN\ba e_2$ has a $3$-cycle. We establish the lemma for when each of $\cN\ba e_1$ and $\cN\ba e_2$ has a $3$-cycle. The other two cases are proved similarly, but are less complicated.

Let $e_1=(u_1, v)$ and $e_2=(u_2, v)$, and let $m$ be a leaf at the end of a tree path starting at $v$. For each $i\in \{1, 2\}$, let $w_i$ and $w'_i$ denote the child of $u_i$ that is not $v$ and the parent of $u_i$, respectively. Since $\cN$ is tree-child and each of $\cN\ba e_1$ and $\cN\ba e_2$ has a $3$-cycle, for each $i$, the vertices $w_i$ and $w'_i$ are tree vertices and the parent of a reticulation $v_i$. For each $i$, let $f_i=(w_i, v_i)$ and $f'_i=(w'_i, v_i)$, and let $m_i$ be a leaf at the end of a tree path starting at $v_i$. Furthermore, let $\ell_i$ be a leaf at the end of a tree path starting at $w_i$. Observe that $\ell_1$, $\ell_2$, $m$, $m_1$, and $m_2$ are distinct as $v$ is normal and neither $e_1$ nor $e_2$ is a shortcut.

Let $\cT\in T(\cN)$. If $\cT$ uses $\{e_2, f_1, f'_2\}$ or $\{e_2, f_1, f_2\}$, then $\cT|\{\ell_1, \ell_2, m, m_1, m_2\}\cong \big\{(\ell_1, m_1), (m, \ell_2, m_2)\big\}$ and $\cT|\{\ell_1, \ell_2, m, m_1, m_2\}\cong \big\{(\ell_1, m_1), (\ell_2, m_2, m)\big\}$, respectively. On the other hand, if $\cT$ uses $\{e_1, f_2, f'_1\}$ or $\{e_1, f_2, f_1\}$, then $\cT|\{\ell_1, \ell_2, m, m_1, m_2\}\cong \big\{(\ell_1, m, m_1), (\ell_2, m_2)\big\}$ and $\cT|\{\ell_1, \ell_2, m, m_1, m_2\}\cong \big\{(\ell_1, m_1, m), (\ell_2, m_2)\big\}$, respectively. It now follows that
$$|T(\cN)|\ge |T(\cN\ba \{e_1, f'_1, f_2\})| + |T(\cN\ba \{e_1, f'_1, f'_2\})| + |T(\cN\ba \{e_2, f'_2, f_1\})| + |T(\cN\ba \{e_2, f'_2, f'_1\})|,$$
and so, by Lemmas~\ref{tnk}(iv) and~\ref{no3cycles}, and the assumption in the statement of the lemma,
$$|T(\cN)|\ge t(n, k-3) + t(n, k-3) + t(n, k-3) + t(n, k-3) > t(n, k).$$
This completes the proof of the lemma.
\end{proof}

\section{Proof of Theorem~\ref{main}} \label{proof}

The proof of Theorem~\ref{main} is inductive and relies on first showing that the theorem holds for all $k\in \{0, 1, 2, 3\}$, the base cases. The next lemma establishes this base case.

\begin{lemma}
Let $\cN$ be a tree-child network with $n$ leaves, $k\in \{0, 1, 2, 3\}$ reticulations, and no $3$-cycles. Then $|T(\cN)|=1$ if $k=0$, $|T(\cN)|=2$ if $k=1$, and
$$|T(\cN)|\ge
\begin{cases}
2, & \mbox{if $k=2$;} \\
3, & \mbox{if $k=3$.}
\end{cases}$$
Furthermore, if $|T(\cN)|=2$ and $k=2$ or $|T(\cN)|=3$ and $k=3$, then $\cN$ is an octopus.
\label{base}
\end{lemma}


\begin{proof} 
Evidently, if $k=0$, then $|T(\cN)|=1$. Furthermore, as $\cN$ has no $3$-cycles, if $k=1$, then $|T(\cN)|=2$, so the lemma holds for $k\in \{0, 1\}$. For the remainder of the proof, we may assume without loss of generality that, amongst all tree-child networks with $n$ leaves, $k\in \{2,3\}$ reticulations, and no $3$-cycles, $|T(\cN)|$ is minimised, in which case, by Lemma~\ref{octopus}, $|T(\cN)|\le t(n, k)$.

First, consider when $k=2$. Let $f = (u,v)$ be a reticulation arc of $\cN$ such that, amongst all reticulation arcs of $\cN$, the vertex $u$ has minimum distance to the root of $\cN$. By Lemma~\ref{tail}, $\cN \ba f$ is tree-child and has no $3$-cycles. Therefore, as $\cN\ba f$ has exactly one reticulation, it follows by the previous base case that $|T(\cN \ba f)|=2$. Furthermore, by Lemma~\ref{octopus}, $|T(\cN)| \leq t(n,2) = 2$, and so
\[ |T(\cN)| = |T(\cN\ba f)| = 2.\]
Thus $f$ is non-essential, and, by Lemma~\ref{tight}, $f$ is the first or last rung of a $2$-tight caterpillar ladder of $\cN$ (in fact, by the choice of $f$, it is the last rung). As $k=2$, this implies that $\cN$ is an octopus, thereby completing the proof for when $k=2$.

Now consider when $k=3$. Again, let $f = (u,v)$ be a reticulation arc of $\cN$ such that, amongst all reticulation arcs of $\cN$, the vertex $u$ has minimum distance to the root of $\cN$. Lemma~\ref{tail} implies that $\cN\ba f$ is tree-child and has no $3$-cycles. Furthermore, as $\cN\ba f$ has exactly two reticulations, it follows by the previous base case that $|T(\cN\ba f)|\ge 2$. Additionally, by Lemma~\ref{octopus}, $|T(\cN)| \leq t(n,3) = 3$, and so $|T(\cN)|\in \{2, 3\}$.

First, suppose that $|T(\cN)| = 2$. Then $|T(\cN)| = |T(\cN \ba f)| =2$, and so $f$ is non-essential. By Lemma~\ref{tight} and the choice of $f$, we have that $f$ is the last rung of a $2$-tight or $3$-tight caterpillar ladder of $\cN$. If $f$ is the last rung of a $3$-tight caterpillar ladder, then $\cN$ is an octopus and $|T(\cN)|=3$, a contradiction as $|T(\cN)| = 2$. Therefore $f$ is the last rung of a $2$-tight caterpillar ladder of $\cN$ with reticulations, $v_1$ and $v_2$ say. Let $v_3$ denote the third reticulation in $\cN$, and let $g_1$ and $g_2$ be the reticulation arcs of $\cN$ directed into $v_3$.

Assume that either $\cN \ba g_1$ or $\cN \ba g_2$ has no $3$-cycle. Without loss of generality, assume $\cN \ba g_1$ has no $3$-cycle. By the previous base case, $|T(\cN \ba g_1)| \geq 2$ and, by assumption, $|T(\cN)|=2$, implying that $|T(\cN)| = |T(\cN \ba g_1)| = 2$, and so $g_1$ is non-essential. Therefore, by Lemma~\ref{tight}, $g_1$ is either the first or last rung of a $2$-tight or $3$-tight caterpillar ladder of $\cN$. However, as $k=3$ and $\cN$ contains a $2$-tight caterpillar ladder with reticulations $v_1$ and $v_2$, and $v_3\not\in \{v_2, v_3\}$, this is not possible.
Thus, for each $i\in \{1, 2\}$, the tree-child network $\cN\ba g_i$ has a $3$-cycle. But neither $3$-cycle involves $v_1$ nor $v_2$ as they are the reticulations of a $2$-tight caterpillar ladder of $\cN$, and so this is also not possible.
In summary, we cannot have $|T(\cN)|=2$, and so $|T(\cN)|=3$. It remains to show that $\cN$ is an octopus. 

Suppose that $|T(\cN)|=3$. By Lemma~\ref{tail}, $\cN\ba f$ is tree-child and has no $3$-cycles. Furthermore, by the previous base case, $|T(\cN \ba f)| \geq 2$ and so, as  $|T(\cN)|=3$, we have $|T(\cN\ba f)| \in \{2, 3\}$. We now distinguish two cases depending on $|T(\cN\ba f)|$:
\begin{enumerate}[{\rm (a)}]
\item If $|T(\cN \ba f)| = |T(\cN)|=3$, then $f$ is non-essential and, by Lemma~\ref{tight} and the choice of $f$, it is the last rung of a $2$-tight or $3$-tight caterpillar ladder of $\cN$. If $f$ is the last rung of a $3$-tight caterpillar ladder of $\cN$, then $\cN$ is an octopus, and so we may assume that $f$ is the last rung of a $2$-tight caterpillar ladder of $\cN$ with reticulations $v_1$ and $v_2$, and core arcs 
\[ \{ (u_2',u_1'), (u_1',u_2), (u_2,u_1), (u_1',v_1), (u_1,v_1), (u_2',v_2), (u_2,v_2) \}. \]
Note that $f = (u, v) = (u_2', v_2)$. Let $v_3$ be the third reticulation of $\cN$ and let $g_1 = (u_3,v_3)$ and $g_2=(u'_3,v_3)$  denote the reticulation arcs of $\cN$ directed into $v_3$. For each $i\in \{1, 2, 3\}$, let $m_i$ denote a leaf at the end of a tree path starting at $v_i$. Furthermore, let $a$, $b$, and $c$ denote a leaf at the end of a tree path starting at $u_1$, $u_3$, and $u'_3$, respectively. Note that $m_1 \neq m_2 \neq m_3$, $a \neq m_1$, $a \neq m_2$, $b \neq m_3$, and $c \neq m_3$. If $a = m_3$, then there is a tree path from $v_3$ to $u_2'$ to $a$ in $\cN$, contradicting the choice of $f$. Thus $a \neq m_3$. In summary,
\begin{align*}
a \neq m_1 \neq m_2 \neq m_3, \, b \neq m_3, \text{ and } c \neq m_3.
\end{align*}

We now consider two subcases:
\begin{enumerate}[{\rm (i)}]
\item Assume that $b\neq c$. By making the appropriate choices of reticulation arcs incident with $v_1$, $v_2$, and $v_3$, it is easily seen that $\cN$ displays phylogenetic trees $\cT_1$, $\cT_2$, $\cT_3$, and $\cT_4$ such that
\begin{align*}
\cT_1 | \{a, m_1, m_2\} &\cong (a, m_1, m_2) \text{ and } \cT_1 | \{b, c, m_3\} \cong (m_3, b, c), \\
\cT_2 | \{a, m_1, m_2\} &\cong (a, m_1, m_2) \text{ and } \cT_2 | \{b, c, m_3\} \cong (m_3, c, b), \\
\cT_3 | \{a, m_1, m_2\} &\cong (a, m_2, m_1) \text{ and } \cT_3 | \{b, c, m_3\} \cong (m_3, b, c), \\
\cT_4 | \{a, m_1, m_2\} &\cong (a, m_2, m_1) \text{ and } \cT_4 | \{b, c, m_3\} \cong (m_3, c, b).
\end{align*}
Since $b\neq c$, we have that $\cT_1$, $\cT_2$, $\cT_3$, and $\cT_4$ are distinct, and so $|T(\cN)|\geq 4$, a contradiction. 

\item On the other hand, if $b = c$, then one of $g_1$ and $g_2$, without loss of generality say $g_2$, is a shortcut. As $\cN$ has no $3$-cycles, the vertex $u'_3$ has a child $w$ that is neither $u_3$ nor $v_3$. Let $\ell$ be a leaf at the end of a tree path that avoids $u_3$ and either starts at $w$ if $w$ is not a parent of a reticulation or starts at such a reticulation. Observe that $\ell\not\in \{b, m_3\}$. By making the appropriate choices of reticulation arcs incident with $v_1$, $v_2$, and $v_3$, it is again easily seen that $\cN$ displays phylogenetic trees $\cT_1'$, $\cT'_2$, $\cT'_3$ and $\cT_4'$ such that
\begin{align*}
\cT_1' | \{a, m_1, m_2\} &\cong (a, m_1, m_2) \text{ and } \cT_1' | \{b,\ell, m_3\} \cong (m_3, b, \ell), \\
\cT_2' | \{a, m_1, m_2\} &\cong (a, m_1, m_2) \text{ and } \cT_2' | \{b, \ell, m_3\} \cong (m_3, \ell, b), \\
\cT_3' | \{a, m_1, m_2\} &\cong (a, m_2, m_1) \text{ and } \cT_3' | \{b, \ell, m_3\} \cong (m_3, b, \ell), \\
\cT_4' | \{a, m_1, m_2\} &\cong (a, m_2, m_1) \text{ and } \cT_4' | \{b, \ell, m_3\} \cong (m_3, \ell, b).
\end{align*}
Since $b$, $\ell$, and $m_3$ are distinct, it follows that $|T(\cN)| \geq 4$, another contradiction. 
\end{enumerate}
Hence if $|T(\cN\ba f)| = |T(\cN)|=3$, then $\cN$ is an octopus.

\item If $|T(\cN\ba f)| = 2$, then, by the previous base case, $\cN\ba f$ is an octopus. In particular, $\cN \ba f$ contains a $2$-tight caterpillar ladder, say $\mathcal{N}'$. 

Assume that $\cN'$ is not a $2$-tight caterpillar ladder of $\cN$. Then $f=(u,v)$ is incident with one of the core arcs of $\cN'$ in $\cN$. If the head $v$ of $f$ is incident with such an arc, then $\cN$ is not tree-child, a contradiction. If instead the tail $u$ of $f$ is incident with one of the core arcs of $\cN'$ in $\cN$, then this contradicts the choice of $f$. So $\cN'$ is a $2$-tight caterpillar ladder of $\cN$ and $f$ is not a core arc of this ladder.

Let $v_1$ and $v_2$ denote the reticulations of the $2$-tight caterpillar ladder $\cN'$ of $\cN$, and assume that its core arcs are given by
\[ \{ (u_2',u_1'), (u_1',u_2), (u_2,u_1), (u_1',v_1), (u_1,v_1), (u_2',v_2), (u_2,v_2) \}. \]
Furthermore, let $v_3$ be the third reticulation of $\cN$ and let $g_1 = (u_3,v_3)$ and $g_2=(u'_3,v_3)$  denote the reticulation arcs of $\cN$ directed into $v_3$. For each $i\in \{1, 2, 3\}$, let $m_i$ denote a leaf at the end of a tree path starting at $v_i$,  and let $a$, $b$, and $c$ denote a leaf at the end of a tree path starting at $u_1$, $u_3$, and $u'_3$, respectively.

If $a \neq m_3$, we can apply the same arguments as in Cases~(a)(i) and~(a)(ii) to conclude that $|T(\cN)| \geq 4$, a contradiction. So we may assume that $a = m_3$, in which case, by the choice of $f$, there is a tree path from $v_3$ to $u_2'$ to $a$ in $\cN$. We now make a few observations. First, $m_1 \neq m_2 \neq a$. Second, if $b = m_1$, the path from $v_1$ to $m_1$ contains a reticulation, contradicting the fact that it is a tree path. Therefore $b \neq m_1$ and, similarly, $b \neq m_2$, $b\neq m_3$, $c \neq m_1$, $c \neq m_2$, and $c_3\neq m_3$. In summary,
\begin{align*}
m_1 \neq m_2 \neq m_3, \, b\not\in \{m_1, m_2, m_3\}, \, c\not\in \{m_1, m_2, m_3\}, \text{ and } a = m_3.
\end{align*}
As in Case~(a), we consider two subcases:
\begin{enumerate}[{\rm (i)}]
\item If $b\neq c$, then, as $b, c\notin \{a, m_1, m_2\}$, it follows as in Case~(a)(i) that $|T(\cN)|\geq 4$, a contradiction.

\item On the other hand, if $b = c$, then, one of $g_1$ and $g_2$, without loss of generality say $g_2$, is a shortcut. Since $\cN$ has no $3$-cycles, the vertex $u'_3$ has a child $w$ that is neither $u_3$ nor $v_3$. Let $\ell$ be a leaf at the end of a tree path avoiding $u'_2$ and either starting at $w$ if $w$ is not the parent of a reticulation or starting at such a reticulation. Observe that such a leaf exists and $\ell\not\in \{b, m_3\}$. By the placement of $v_3$, we have $\ell \notin \{a, m_1, m_2\}$. It is now easily checked that $\cN$ displays phylogenetic trees $\cT_1''$, $\cT''_2$, $\cT''_3$ and $\cT_4''$ such that
\begin{align*}
\cT_1'' | \{a, m_1, m_2\} &\cong (a, m_1, m_2) \text{ and } \cT_1'' | \{a, b,\ell\} \cong (a, b, \ell), \\
\cT_2'' | \{a, m_1, m_2\} &\cong (a, m_1, m_2) \text{ and } \cT_2'' | \{a, b, \ell\} \cong (b, \ell, a), \\
\cT_3'' | \{a, m_1, m_2\} &\cong (a, m_2, m_1) \text{ and } \cT_3'' | \{a, b, \ell\} \cong (a, b, \ell), \\
\cT_4'' | \{a, m_1, m_2\} &\cong (a, m_2, m_1) \text{ and } \cT_4'' | \{a, b, \ell\} \cong (b, \ell, a).
\end{align*}
Since $a$, $b$, and $\ell$ are distinct, it follows that $|T(\cN)| \geq 4$, a contradiction. 
\end{enumerate}
Hence $|T(\cN \ba f)|\neq 2$.
\end{enumerate}
This completes the proof of the lemma.
\end{proof}

\begin{proof}[Proof of Theorem~\ref{main}]
Without loss of generality, we may assume that, amongst all tree-child networks with $n$ leaves, $k$ reticulations, and no $3$-cycles, $|T(\cN)|$ is minimised, in which case, by Lemma~\ref{octopus}, $|T(\cN)|\le t(n, k)$. By Lemma~\ref{base}, the theorem holds for $k\in \{0, 1, 2, 3\}$. Now suppose that $k\ge 4$, and so $n\ge 5$, and the theorem holds for all tree-child networks with $n$ leaves, at most $k-1$ reticulations, and no $3$-cycles.

Let $p_B$ be a tree vertex of $\cN$ that is a parent of a reticulation, say $p_A$, so that, amongst all such tree vertices, $p_B$ has maximum distance to the root. Let $A$ and $B$ denote the leaf sets of the pendant subtrees of $\cN$ obtained by deleting the outgoing arc of $p_A$ and the outgoing arc of $p_B$ that is not $(p_B, p_A)$, respectively. By maximality, $A$ and $B$ are well defined. Let $a\in A$ and $b\in B$, and let $e_1$ and $e_2$ denote the reticulation arcs of $\cN$ directed into $p_A$, where $e_1=(p_B, p_A)$. Note that $e_1$ is not a shortcut. Let $q_A$ denote the parent of $p_A$ that is not $p_B$, that is $e_2=(q_A, p_A)$. By Lemma~\ref{normal} and the minimality and induction assumptions, $p_A$ is not normal and so $e_2$ is a shortcut. With this setup, the remainder of the proof is to show that $e_1$ is the first rung of a $2$-tight or $3$-tight caterpillar ladder of $\cN$, and then use induction to show that $\cN$ is an octopus.

Let $P_u=q_A, u_1, u_2, \ldots, u_r, p_B$ be a directed path from $q_A$ to $p_B$ in $\cN$. Since $\cN$ has no $3$-cycles, we have $r\ge 1$. We next show that, for all $i\in \{1, 2, \ldots, r\}$, the vertex $u_i$ is a tree vertex and the parent of a reticulation. Note that this will imply that $P_u$ is the unique directed path from $q_A$ to $p_B$. In fact, we eventually show that $r=1$.

Consider $u_1$. Since $\cN$ is tree-child, $u_1$ is not a reticulation. Assume that there is a tree path from $u_1$ to a leaf $\ell$ avoiding $p_B$. Let $\cT\in T(\cN)$. Then $\cT$ uses $e_1$ if and only if $\cT|\{a, b, \ell\}\cong (a, b, \ell)$, and $\cT$ uses $e_2$ if and only if either $\cT|\{a, b, \ell\}\cong (a, \ell, b)$ or $\cT|\{a, b, \ell\}\cong (b, \ell, a)$. In particular, if $\cT_1$, $\cT_2\in T(\cN)$, and $\cT_1$ uses $e_1$ and $\cT_2$ uses $e_2$, then $\cT_1$ and $\cT_2$ are not isomorphic (an argument that we use repeatedly throughout the proof). Thus
$$|T(\cN)| = |T(\cN\ba e_2)| + |T(\cN\ba e_1)|.$$
If $\cN\ba e_2$ has a $3$-cycle and $f$ is a reticulation arc of this $3$-cycle, then, by Lemma~\ref{no3cycles}, $\cN\ba \{e_2, f\}$ has no $3$-cycles. Therefore, as $\cN\ba e_1$ has no $3$-cycle, it follows by induction and Lemma~\ref{tnk}(v) that
$$|T(\cN)|\ge t(n, k-1) + t(n, k-2) > t(n, k),$$
a contradiction to the minimality of $|T(\cN)|$. Thus there is no such tree path from $u_1$. Instead, as $\cN$ is tree-child, all tree paths from $u_1$ to a leaf traverse $p_B$. This implies that $P_u$ has no reticulations and, for all $i\in \{1, 2, \ldots, r\}$, the vertex $u_i$ is a parent of a reticulation, $v_i$ say. For each $i\in \{1, 2, \ldots, r\}$, let $m_i$ denote a leaf at the end of a tree path starting at $v_i$ and observe that $m_i\not\in A\cup B$. Furthermore, let $u'_i$ denote the second parent of $v_i$, and let $f_i=(u_i, v_i)$ and $f'_i=(u'_i, v_i)$. By Lemma~\ref{normal} and the minimality and induction assumptions, $v_i$ is not normal for all $i$, and so either $f_i$ or $f'_i$ is a shortcut.

\begin{sublemma}
For all $i\in \{1, 2, \ldots, r\}$, the arc $f'_i$ is a shortcut and there is a directed path from $u'_i$ to $q_A$.
\label{path1}
\end{sublemma}

Assume that, for some $i$, both $u_i$ and $u'_i$ lie on $P_u$. Let $\cT\in T(\cN)$. If $\cT$ uses $e_1$, then $\cT|\{a, b, m_i\}\cong (a, b, m_i)$. But, if $\cT$ uses $e_2$, then $\cT|\{a, b, m_i\}\cong (b, m_i, a)$. Therefore
$$|T(\cN)|\ge |T(\cN\ba e_2)| + |T(\cN\ba e_1)|.$$
Noting that $\cN\ba e_2$ may contain a $3$-cycle, in which case, by Lemma~\ref{no3cycles}, deleting one further arc results in a tree-child network with no $3$-cycles, it follows by induction and Lemma~\ref{tnk}(v) that
$$|T(\cN)|\ge t(n, k-2) + t(n, k-1) > t(n, k).$$
Hence, for all $i\in \{1, 2, \ldots, r\}$, the vertex $u'_i$ does not lie on $P_u$ and so, as $v_i$ is not normal, $(u'_i, v_i)$ is a shortcut and there is a directed path from $u'_i$ to $q_A$ for all $i$. This proves (\ref{path1}).

\begin{sublemma}
$r=1$.
\label{r=1}
\end{sublemma}

Assume that $r\ge 2$. Let $\cT\in T(\cN)$. If $\cT$ uses $e_1$, then $A\cup B$ is a cluster of $\cT$. Also, if $\cT$ uses $\{e_2, f_1, f'_2\}$ or $\{e_2, f'_1, f_2\}$, then $\cT|\{a, b, m_1, m_2\}\cong (b, m_1, a, m_2)$ and $\cT|\{a, b, m_1, m_2\}\cong (b, m_2, a, m_1)$, respectively. Observing that if $\cT|\{a, b, m_1, m_2\}$ is isomorphic to either $(b, m_1, a, m_2)$ or $(b, m_2, a, m_1)$, then $A\cup B$ is not a cluster of $\cT$, it follows that
$$|T(\cN)|\ge |T(\cN\ba e_2)| + |T(\cN\ba \{e_1, f'_1, f_2\})| + |T(\cN\ba \{e_1, f_1, f'_2\})|.$$
Say $\cN\ba e_2$ has no $3$-cycle. If $k\ge 6$, then, as $\cN\ba \{e_1, f_1\}$ and $\cN\ba \{e_1, f_2\}$ have no $3$-cycles, it follows by induction and Lemmas~\ref{tnk}(v) and~\ref{no3cycles} that
\begin{align*}
|T(\cN)| & \ge t(n, k-1) + t(n, k-4) + t(n, k-4) \\
& = t(n, k-1) + t(n, k-2) \\
& > t(n, k).
\end{align*}
Furthermore, if $k=4$, then
$$|T(\cN)|\ge t(n, 3) + t(n, 0) + t(n, 0) = 3 + 1 + 1 > 4 = t(n, 4)$$
while, if $k=5$, then
$$|T(\cN)|\ge t(n, 4) + t(n, 1) + t(n, 1) = 4 + 2 + 2 > 6 = t(n, 5).$$
On the other hand, if $\cN\ba e_2$ has a $3$-cycle, then, by construction, this $3$-cycle contains the reticulation arcs $f_1$ and $f'_1$, and so $\cN\ba f'_1$ has no $3$-cycles. Thus, if $k\ge 6$, then, by induction and Lemmas~\ref{tnk}(v) and~\ref{no3cycles},
\begin{align*}
|T(\cN)| & \ge t(n, k-2) + t(n, k-3) + t(n, k-4) \\
& > t(n, k-2) + t(n, k-2) \\
& = t(n, k).
\end{align*}
Also, if $k=4$, then
$$|T(\cN)|\ge t(n, 2) + t(n, 1) + t(n, 0) = 2 + 2 + 1 > 4 =t(n, 4)$$
while, if $k=5$, then
$$|T(\cN)|\ge t(n, 3) + t(n, 2) + t(n, 1) = 3 + 2 + 2 > 6 = t(n, 5).$$
These contradictions to the minimality of $|T(\cN)|$ imply that $r=1$, thereby proving (\ref{r=1}).

To simplify notation, set $u=u_1$, $u'=u'_1$, $v=v_1$, $m_u=m_1$, $f_u=f_1$, and $f'_u=f'_1$. 

\begin{sublemma}
If $(u', q_A)$ is an arc of $\cN$, then $\cN$ is an octopus.
\label{oct1}
\end{sublemma}

Assume that $(u', q_A)$ is an arc of $\cN$. Let $\cT\in T(\cN)$. If $\cT$ uses $\{e_1, f_u\}$, then $\cT|\{a, b, m_u\}\cong (a, b, m_u)$, while if $\cT$ uses $\{e_2, f_u\}$, then $\cT|\{a, b, m_u\}\cong (b, m_u, a)$. So
$$|T(\cN)|\ge |T(\cN\ba \{e_2, f'_u\})| + |T(\cN\ba \{e_1, f'_u\})|.$$
By induction and Lemma~\ref{tnk}(v) and~\ref{no3cycles},
$$|T(\cN)|\ge t(n, k-2) + t(n, k-2) = t(n, k).$$
Therefore, as $|T(\cN)|\le t(n, k)$,
\begin{align}
|T(\cN)| = 2\cdot t(n, k-2) = t(n, k).
\label{eqn1}
\end{align}
We now show that $\cN$ is an octopus. Since $\{u', q_A, u, p_B, p_A, v\}$ induces the core of a $2$-tight caterpillar ladder of $\cN$, it follows by Lemma~\ref{tight} that $f'_u$ is non-essential. Thus
$$|T(\cN)| = |T(\cN\ba f'_u)|.$$
In turn, it is now easily checked that
$$|T(\cN)| = 2|T(\cN\ba \{f'_u, e_2\})|$$
and so, by (\ref{eqn1}),
$$|T(\cN\ba \{f'_u, e_2\})| = t(n, k-2).$$
Therefore, by induction, $\cN\ba \{f'_u, e_2\}$ is an octopus, and so, by construction, $\cN$ is an octopus. This proves (\ref{oct1}).

Now assume that $(u', q_A)$ is not an arc of $\cN$. Let $P_t=u', t_1, t_2, \ldots, t_s, q_A$ be a directed path from $u'$ to $q_A$ in $\cN$. Similar to before, we will show that $s=1$ and $t_1$ is the parent of a reticulation. Consider $t_1$. Since $\cN$ is tree-child, $t_1$ is not a reticulation. Say that there is a tree path from $t_1$ to a leaf $\ell$ avoiding $q_A$. Let $\cT\in T(\cN)$. If $\cT$ uses $\{e_1, f_u\}$, then $\cT|\{a, b, m_u, \ell\}\cong (a, b, m_u, \ell)$. If $\cT$ uses either $\{e_2, f_u\}$ or $\{e_2, f'_u\}$, then $\cT|\{a, b, m_u, \ell\}\cong (b, m_u, a, \ell)$ and $\cT|\{a, b, m_u, \ell\}\in \big\{(a, b, \ell, m_u), \{(a, b), (m_u, \ell)\}\big\}$, respectively. So
$$|T(\cN)|\ge |T(\cN\ba \{e_2, f'_u\})| + |T(\cN\ba \{e_1, f'_u\})| + |T(\cN\ba \{e_1, f_u\})|.$$
Now $\cN\ba e_2$ has no $3$-cycle as $u'$ is not the parent of $q_A$. Therefore, if $k\ge 5$, it follows by induction and Lemmas~\ref{tnk}(v) and~\ref{no3cycles} that
\begin{align*}
|T(\cN)| & \ge t(n, k-3) + t(n, k-3) + t(n, k-2) \\
& = t(n, k-1) + t(n, k-2) \\
& > t(n, k),
\end{align*}
a contradiction. Furthermore, if $k=4$, then
$$|T(\cN)| \ge t(n, 1) + t(n, 1) + t(n, 2) = 2 + 2 + 2 > 4 = t(n, k),$$
another contradiction. Thus there is no such tree path from $t_1$, and so all tree paths from $t_1$ to a leaf traverse $q_A$. In turn, this implies that $P_t$ has no reticulations, and so $P_t$ consists of tree vertices and, for all $i\in \{1, 2, \ldots, s\}$, the vertex $t_i$ is a parent of a reticulation, $w_i$ say. For each $i\in \{1, 2, \ldots, s\}$, let $m_i$ denote the leaf at the end of a tree path starting at $w_i$. Also let $t'_i$ denote the second parent of $w_i$, and let $g_i=(t_i, w_i)$ and $g'_i=(t'_i, w_i)$. By Lemma~\ref{normal}, and the minimality and induction assumptions, either $g_i$ or $g'_i$ is a shortcut for all $i$.

\begin{sublemma}
For all $i\in \{1, 2, \ldots, s\}$, the arc $g'_i$ is a shortcut and there is a directed path from $t'_i$ to $u'$.
\label{path2}
\end{sublemma}

Assume that, for some $i$, both $t_i$ and $t'_i$ lie on $P_t$. Let $\cT\in T(\cN)$. If $\cT$ uses $\{e_2, f'_u\}$, then $\cT|\{a, b, m_u, m_i\}\cong (a, b, m_i, m_u)$. Furthermore, if $\cT$ uses $\{e_1, f_u\}$, then $\cT|\{a, b, m_u, m_i\}\cong (a, b, m_u, m_i)$, while if $\cT$ uses $\{e_2, f_u\}$, then $T|\{a, b, m_u, m_i\}\cong (b, m_u, a, m_i)$. Thus
$$|T(\cN)|\ge |T(\cN\ba \{e_1, f_u\})| + |T(\cN\ba \{e_2, f'_u\})| + |T(\cN\ba \{e_1, f'_u\})|.$$
Noting that $N\ba f'_u$ may contain a $3$-cycle but $\cN\ba e_2$ does not contain a $3$-cycle, it follows by induction and Lemmas~\ref{tnk}(v) and~\ref{no3cycles} that, if $k\ge 5$, then
\begin{align*}
|T(\cN)| & \ge t(n, k-2) + t(n, k-3) + t(n, k-3) \\
& = t(n, k-2) + t(n, k-1) \\
& > t(n, k),
\end{align*}
a contradiction. Also, if $k=4$, then
$$|T(\cN)| \ge t(n, 2) + t(n, 1) + t(n, 1) = 2 + 2 + 2  > 4 =t(n, 4),$$
another contradiction. Thus, for all $i\in \{1, 2, \ldots, s\}$, the vertex $t'_i$ does not lie on $P_t$, and so $g'_i$ is a shortcut and  there is a directed path from $t'_i$ to $u'$ for all $i$. This proves (\ref{path2}).

\begin{sublemma}
$s=1$.
\label{s=1}
\end{sublemma}

Assume that $s\ge 2$. Let $\cT\in T(\cN)$. If $\cT$ uses $e_1$, then $A\cup B$ is a cluster of $\cT$. Also, if $\cT$ uses $\{e_2, f_u\}$, then $\cT|\{a, b, m_u\}\cong (b, m_u, a)$. Furthermore, if $\cT$ uses $\{e_2, f_u, g_1, g'_2\}$ or $\{e_2, f_u, g'_1, g_2\}$, then $\cT|\{a, b, m_u, m_1, m_2\}\cong (b, m_u, a, m_1, m_2)$ and $\cT|\{a, b, m_u, m_1, m_2\}\cong (b, m_u, a, m_2, m_1)$, respectively. Therefore
\begin{align}
|T(\cN)| & \ge |T(\cN\ba e_2)| + |T(\cN\ba \{e_1, f'_u\})|
\label{eqn2}
\end{align}
and
\begin{align}
|T(\cN)|\ge |T(\cN\ba e_2)| + |T(\cN\ba \{e_1, f'_u, g'_1, g_2\})| + |T(\cN\ba \{e_1, f'_u, g_1, g'_2\}|.
\label{eqn3}
\end{align}
If $\cN\ba f'_u$ has a $3$-cycle, then, by construction, this $3$-cycle contains $g_1$ and $g'_1$, in which case, by Lemma~\ref{no3cycles}, neither $\cN\ba \{f'_u, g_1\}$ nor $\cN\ba \{f'_u, g'_1\}$ has a $3$-cycle. Moreover, none of $\cN\ba e_2$, $\cN\ba g_1$, and $\cN\ba g_2$ has a $3$-cycle. Therefore, if $k$ is even, then, by (\ref{eqn2}), induction, and Lemmas~\ref{tnk}(vi) and~\ref{no3cycles},
\begin{align*}
|T(\cN)| & \ge t(n, k-1) + t(n, k-3) \\
& > t(n, k),
\end{align*}
a contradiction. If $k\ge 7$ and odd, then, by (\ref{eqn3}), induction, Lemma~\ref{tnk}(v) and (vi), and Lemma~\ref{no3cycles},
\begin{align*}
|T(\cN)|& \ge t(n, k-1) + t(n, k-4) + t(n, k-5) \\
& > t(n, k-1) + t(n, k-3) \\
& = t(n, k),
\end{align*}
while if $k=5$, then
$$|T(\cN)| \ge t(n, 4) + t(n, 1) + t(n, 0) = 4 + 2 + 2 > 6 = t(n, 5).$$
These contradictions imply that $s=1$, thereby proving (\ref{s=1}).

To simplify (for the last time) the notation, set $t=t_1$, $t'=t'_1$, $w=w_1$, $m_t=m_1$, $g_t=g_1$, and $g'_t=g'_1$.

\begin{sublemma}
If $(t', u')$ is an arc of $\cN$, then $\cN$ is an octopus.
\label{oct2}
\end{sublemma}

Assume that $(t', u')$ is an arc of $\cN$. Let $\cT\in T(\cN)$. If $\cT$ uses $e_1$, then $A\cup B$ is a cluster of $\cT$, while if $\cT$ uses $\{e_2, f_u\}$, then $\cT|\{a, b, m_u\}\cong (b, m_u, a)$. So
$$|T(\cN)| \ge |T(\cN\ba e_2)| + |T(\cN\ba \{e_1, f'_u\}|.$$
Since $\cN\ba \{e_2, f'_u, g'_t\}$ has no $3$-cycles by Lemma~\ref{no3cycles}, it follows by induction that
$$|T(\cN)| \ge t(n, k-1) + t(n, k-3).$$
If $k$ is even, then, by Lemma~\ref{tnk}(vi), we have $|T(\cN)| > t(n, k)$, a contradiction. So say $k\geq 5$ and odd. Then, as $|T(\cN)|\le t(n, k)$,
$$|T(\cN)| = t(n, k-1) + t(n, k-3) = t(n, k).$$
We now show that $\cN$ is an octopus. Since $\{t', u', t, q_A, u, p_B, w, v, p_A\}$ induces the core of a $3$-tight caterpillar ladder of $\cN$, it follows by Lemma~\ref{tight} that $g'_t$ is non-essential. Thus
$$|T(\cN)| = |T(\cN\ba g'_t)|.$$
In turn, it is easily checked that 
$$|T(\cN)| = 3\cdot |T(\cN\ba \{g'_t, f'_u, e_2\}|.$$
Therefore, by induction,
$$t(n, k-1) + t(n, k-3) = |T(\cN)| = 3\cdot |T(\cN\ba \{g'_t, f'_u, e_2\})|\ge 3\cdot t(n, k-3),$$
that is, as $k \geq 5$, and thus, by Lemma~\ref{tnk}(v), $t(n, k-1) = 2\cdot t(n, k-3)$, we have
$$|T(\cN\ba \{g'_t, f'_u, e_2\})| = t(n, k-3).$$
Hence, by induction, $\cN\ba \{g'_t, f'_u, e_2\}$ is an octopus with an even number of reticulations. It follows by construction that $\cN$ is an octopus, completing the proof of (\ref{oct2}).

Thus we may now assume that $(t', u')$ is not an arc. Let $\cT\in T(\cN)$. If $\cT$ uses $e_1$, then $A\cup B$ is cluster of $\cT$, while if $\cT$ uses $\{e_2, f_u\}$, then $A\cup B$ is not a cluster of $\cT$. Thus
$$|T(\cN)|\ge |T(\cN\ba e_2)| + |T(\cN\ba \{e_1, f'_u\})|.$$
Since $(t', u')$ is not an arc, $\cN\ba \{e_1, f'_u\}$ has no $3$-cycles and so, by induction and Lemma~\ref{tnk}(v),
$$|T(\cN)|\ge t(n, k-1) + t(n, k-2) > t(n, k).$$
This last contradiction completes the proof of the theorem.
\end{proof}

\section*{Acknowledgments} The authors were supported by the New Zealand Marsden Fund.

\bibliographystyle{abbrvnat}
\bibliography{References.bib}

\end{document}